\documentclass[11pt]{article}
\usepackage{natbib}
\usepackage{graphicx}
\usepackage{latexsym}
\usepackage[left=1.15in,right=1.15in,bottom=1.5in,top=1.5in]{geometry}
\usepackage{amsmath,amssymb,dsfont}
\usepackage{amsfonts}
\usepackage{amsthm}
\usepackage{color}
\RequirePackage{xcolor}[2.11]
\colorlet{inlinkcolor}{blue!10!black}
\colorlet{exlinkcolor}{red!40!black}
\colorlet{reviewcolor}{black!50}

\usepackage{hyperref}
\hypersetup{
 colorlinks = true,
 allcolors = exlinkcolor,
 urlcolor = exlinkcolor,
}
\usepackage{printlen}
\newtheorem{algo}{Algorithm}
\DeclareMathOperator*{\argmin}{arg\,min}
\DeclareMathOperator{\dist}{dist}

\DeclareMathOperator*{\dom}{dom}

\def\ds{\displaystyle}

\def\cH{{\mathcal H}}
\def\bR{{\mathds R}}
\def\cK{{\mathcal K}}

\def\cX{{\mathcal X}}

\def\cL{{\mathcal L}}

\def\({\begin{equation}}
\def\){\end{equation}}
\def\keywords#1{\par
	\vspace*{8pt}
	{ {\leftskip18pt\rightskip\leftskip
	\noindent{\it Keywords}\/:\ #1\par}}\par}
	\newtheorem{definition}{Definition}
\newtheorem{assumption}{Assumption}

\newtheorem{theorem}{Theorem}

\newtheorem{remark}{Remark}
\numberwithin{equation}{section}
\numberwithin{lemma}{section}
\numberwithin{definition}{section}
\numberwithin{assumption}{section}
\numberwithin{theorem}{section}
\numberwithin{proposition}{section}
\numberwithin{corollary}{section}
\numberwithin{remark}{section}
\numberwithin{equation}{section}

\title{\bf\Large On the Linear and Asymptotically Superlinear
Convergence Rates of the Augmented Lagrangian Method
with a Practical Relative Error Criterion}

\author{Xin-Yuan Zhao\thanks{College of Applied Sciences, Beijing University of Technology, Beijing 100022, P. R. China (\url{xyzhao@bjut.edu.cn}).The research of this author was supported by the National Natural Science Foundation of China (11871002) and the General Program of Science and Technology of Beijing Municipal Education Commission.}
\qquad
Liang Chen\thanks{School of Mathematics, Hunan University, Changsha, P.R. China (\url{chl@hnu.edu.cn}). The research of this author was supported by the National Natural Science Foundation of China (11801158, 11871205), the Hunan Provincial Natural Science Foundation of China (2019JJ50040), and the Fundamental Research Funds for the Central Universities in China.}}
\date{\today}%

\begin{document}
\maketitle

\begin{abstract}
In this paper, we conduct a convergence rate analysis of the augmented Lagrangian method with a practical relative error criterion designed in Eckstein and Silva [Math. Program., 141, 319--348 (2013)] for convex nonlinear programming problems.
We show that under a mild local error bound condition, this method admits locally a Q-linear rate of convergence.
More importantly, we show that the modulus of the convergence rate is inversely proportional to the penalty parameter. That is, an asymptotically superlinear convergence is obtained if the penalty parameter used in the algorithm is increasing to infinity, or an arbitrarily Q-linear rate of convergence can be guaranteed if the penalty parameter is fixed but it is sufficiently large.
Besides, as a byproduct, the convergence, as well as the convergence rate, of the distance from the primal sequence to the solution set of the problem is obtained.
\end{abstract}


\section{Introduction}

The {\it fast} (asymptotically superlinear or arbitrarily Q-linear) local convergence rate is perhaps computationally the most crucial property of the augmented Lagrangian method.
This topic has long been discussed ever since the birth of the augmented Lagrangian method in \cite{powell}, and proceeded thereafter in \cite{rockafellar}, \cite{tre73}, \cite{ber76}, \cite{conn}, \cite{ito}, \cite{cont93}, \cite{zhang}, \cite{sun07}, \cite{chenjota} and \cite{cui}, for various problem settings.
Generally speaking, it has been observed that if the augmented Lagrangian method achieves a Q-linear convergence rate, the rate should be inversely proportional to the penalty parameter in the algorithm, i.e., the corresponding modulus should tend to zero as the penalty parameter increases to infinity.
As a result, both theoretically and practically, large (but not too large to influence the numerical stability) penalty parameters are desirable in implementing the augmented Lagrangian method.
In fact, the fast local convergence of the augmented has been exploited in many efficient solvers for large-scale convex optimization problems, e.g., \cite{zhaoxy}, \cite{yanglq}, \cite{chench}, \cite{lixd}, \cite{lixd2}, \cite{lixd3}, \cite{zhangyj} and \cite{liu2019}, to name a few.

The main objective of this paper is to investigate whether such a fast local Q-linear convergence is still preserved for the inexact augmented Lagrangian method with a practical relative error criterion proposed by \cite{eckstein13}.
This criterion is of some practical interests since it is checkable whenever the gradient (or the subgradient) of the augmented Lagrangian function is computable.
Moreover, this criterion is relative in the sense that only {\it one} tolerance parameter, instead of a summable sequence of nonnegative real numbers, such as that in \cite{rockafellar}, is used to control the error.
As was reported in \cite{eckstein13}, this approximation criterion improves the augmented Lagrangian method and therefore has significant practical value, but the corresponding algorithm can not be interpreted, at least currently, as an application of the proximal point algorithm (PPA) in \cite{rockafellar76} or its inexact variants. 
As a result, the convergence analysis in \cite{eckstein13} is not based on the convergence properties of the PPAs, which is
very different from the convergence analysis for the augmented Lagrangian method in \cite{rockafellar76}, where the convergence properties, including the fast Q-linear convergence rate, are heavily dependent on those of PPAs.
Consequently, the convergence properties of the augmented Lagrangian method with the relative error criterion in \cite{eckstein13} can not be obtained by directly applying some known results.
Therefore, it is of great interest to know if the algorithm in \cite{eckstein13} also admits a {\em fast} (local) convergence and if the modulus is inversely proportional to the penalty parameter.

The classic augmented Lagrangian method, also known as the method of multipliers, was initiated by \cite{hestenes69} and \cite{powell} for solving equality constrained nonlinear programming problems.
For general nonlinear programming problems, one may refer to the monograph of \cite{berbook} and the references therein, where the details for various aspects of the augmented Lagrangian methods were systematically provided.
In the context of convex programming, as was illustrated in \cite{rockafellar}, the classic augmented Lagrangian method can be viewed as a PPA applied to the dual of the given problem.
Moreover, in the seminal papers \cite{rockafellar} and \cite{rockafellar76}, the author has studied both the convergence and the {\em fast} local convergence rate of the augmented Lagrangian method, even if in the subproblems are approximately solved with a few summable sequences of nonnegative real numbers controlling the errors.
In fact, the convergence rate in \cite{rockafellar} for the augmented Lagrangian method is inherited from the convergence rate of the inexact PPA established in \cite{rockafellar76}, where the errors are also controlled by a summable sequence of nonnegative real numbers.

Along a different line, inexact PPAs with relative error criteria, in which summable sequences of real numbers are no longer necessary, also have been well studied in \cite{solodov99,solodov992,solodov00}.
Moreover, the corresponding local Q-linear convergence rate was also available in \cite{solodov992}.
There are many concrete instances of applications of these inexact PPAs, but direct applications of them to augmented Lagrangian methods would bring conditions that are not practically verifiable.
Fortunately, in \cite{eckstein13}, the authors developed a practical relative error criterion for approximately minimizing the subproblems in the augmented Lagrangian method.
This criterion is precisely implementable in the sense that all the information needed for checking it can be obtained directly at every candidate approximate solution to the subproblems, instead of using certain values such as the infimum of the objective functions to the subproblems in \cite{rockafellar}, which are generally not available\footnote{It wa until recently that practical surrogates of such criteria has been developed for a class of convex composite conic programming problems in \cite{cui}.}.
In \cite{eckstein13}, an informative discussion on the design of this relative error criterion for the augmented Lagrangian method was provided, and the comparisons of this criterion to those of \cite{solodov99,solodov992,solodov00} for inexact PPAs are also elaborated.

To be consistent with, and even more general\footnote{In \cite{eckstein13}, $\cX\equiv\bR^n$, and the objective function $f$ is assumed to be either a continuous differentiable convex function, or the sum of a continuous differentiable convex function and the indicator function of a closed convex box in $\bR^n$. However, the algorithm in \cite{eckstein13} and all the results established in \cite{eckstein13} and \cite{alves} are valid for problem \eqref{prob}, with the corresponding gradients and normal cones being replaced by subgradients.} than, that of \cite{eckstein13}, the problem setting of this paper is as follows
\(
\label{prob}
\begin{array}{cl}
\ds\min_{x\in\cX} &f(x)\\[2mm]
\mbox{s.t.}& h(x)=0, \\[2mm]
&g(x) \le 0,
\end{array}
\)
where $\cX$ is a finite dimensional real Hilbert space endowed with an inner product $\langle\cdot,\cdot\rangle$ and $\|\cdot\|$ is the corresponding norm,
$f:\cX\to(-\infty,+\infty]$ is a closed proper convex function, $h:\cX\to\bR^{m_1}$ is an affine mapping,
and $g:\cX\to\bR^{m_2}$ is a nonlinear mapping, i.e., $g(x)=(g_1(x);\ldots; g_{m_2}(x))$, with each $g_i:\cX\to(-\infty,\infty)$, $i=1,\ldots,m_2$ being a continuously differentiable convex function.

In fact, the original convergence results and the corresponding convergence analysis are not sufficient for analyzing the convergence rate of the algorithm in \cite{eckstein13}. However, fortunately, the convergence analysis has been further improved in \cite{alves}, where a novel technique for treating F\'ejer monotone sequence in product spaces was developed.
This partially paves the way for studying the convergence rate of the algorithm in \cite{eckstein13}.

Motivated by the above expositions and the fact that allowing the subproblems being solved approximately would further contribute to the efficiency of the augmented Lagrangian method,
we are interested in further exploring the convergence properties of the augmented Lagrangian method with the relative error criterion developed in \cite{eckstein13} for solving problem \eqref{prob}.
In this paper, we will show that, under a mild local error bound condition, this algorithm also admits a {\em fast} Q-linear local convergence in the sense that the convergence rate of the dual sequence is Q-linear and the modulus of this rate tends to zero if the penalty parameter increases to infinity.
Besides, we show that such a local error bound condition is also sufficient to guarantee the convergence of the distance from the primal sequence, generated by the algorithm, to the optimal solution set of problem \eqref{prob}.
Here, we should emphasize that neither \cite{eckstein13} nor \cite{alves} has established the convergence of the primal sequence.

This remaining parts of this paper are organized as follows.
In Sect. \ref{sec:pre}, we provide the notation and preliminaries that will be used throughout this paper. In Sect. \ref{sec:algo}, we present the inexact augmented Lagrangian method for solving problem \eqref{prob} with the practical relative error criterion developed in \cite{eckstein13}. Also, we will summarize the corresponding convergence properties from \cite{eckstein13} and \cite{alves}.
Here, we also present some useful equalities and inequalities for further use.
In Sect. \ref{sec:conv}, we introduce a local error bound condition and establish the fast local convergence of the algorithm presented in Sect. \ref{sec:algo}, as well as the global convergence of the primal sequence.
We conclude this paper in Sect. \ref{sec:conclusion}.

\subsection*{Notation}
Let $\cH$ be a finite dimensional real Hilbert space endowed with the inner product $\langle\cdot,\cdot\rangle$.
We use $\|\cdot\|$ to denote the norm induced by this inner product.
For any given $z\in\cH$, we use ${\mathds B}_\epsilon(z)$ to denote the closed ball centring at $z$ with the radius $\epsilon\ge 0$, i.e.,
$$
{\mathds B}_\epsilon(z):=\{z'\in\cH\mid \|z-z'\|\le \epsilon\}.
$$
Let $C\subset\cH$ be a nonempty closed convex set. The projection of a vector $z\in\cH$ onto the set $C$ is defined by
\[
\Pi_C(z):=\argmin_{z'\in C}\left\{\frac{1}{2}\|z'-z\|^2\right\},
\]
while the distance from $z\in\cH$ to the set $C$ is defined by
\[
\dist (z,C):=\|z-\Pi_C(z)\|.
\]

Let $\theta:\cH\to(-\infty,+\infty]$ be an arbitrary closed proper convex function, we use $\dom\, \theta$ to denote its effective domain, i.e.,
$$
\dom\, \theta:=\{z\in\cH\mid \theta(z)<+\infty\}$$
and $\partial\theta$ to denote its subdifferential mapping, i.e.,
$$
\partial \theta(z):=\{
\gamma \in\cH \mid \theta(z')-\theta(z)\ge\langle \gamma, z'-z\rangle,\ \forall z'\in\cH
\},\quad\forall z\in\cH.
$$
Moreover, if $\theta$ is continuously differentiable at $z\in\cH$, one has $\partial\theta(z)=\{\nabla\theta(z)\}$, where $\nabla\theta(z)$ is the gradient of $\theta$ at $z$.

If $\cH\equiv\bR^l$, i.e., the $l$-dimensional real vector space, we use the dot product as the inner product, i.e.,  $\langle z,z'\rangle:=z^Tz'$, so that $\|\cdot\|$ is the conventional $\ell_2$-norm. In this case we define
$$\max\{z,z'\}:=\{\bar z\mid \bar z_i=\max\{z_i,z'_i\}, i=1\ldots,l\},\quad\forall z,z'\in\cH, 
$$ i.e., the component-wise maximum between $z$ and $z'$. The component-wise minimum is also defined in the same fashion.

\section{Preliminaries}
\label{sec:pre}
Generally, the notation and definitions used in this paper are the same as those used in \cite{eckstein13} and \cite{alves}. In fact, they are consistent with those in \cite{rocbook,rocbookcg}.

The Lagrangian function $\cL:\cX\times\bR^{m_1}\times\bR^{m_2}\to[-\infty,\infty]$ of problem \eqref{prob} is defined by
\[
\cL(x;\lambda,\mu):=\left\{
\begin{array}{ll}
f(x)+\langle \lambda, h(x)\rangle+\langle \mu, g(x)\rangle,\quad &\mbox{if }\mu\ge 0,\\[2mm]
-\infty, & \mbox{otherwise}.
\end{array}
\right.
\]
Then, the dual of problem \eqref{prob} is given by
\(
\label{dual}
\max_{\lambda\in\bR^{m_1},\mu\in\bR^{m_2}}\left\{ d(\lambda,\mu):=\inf_{x\in\cX}\cL(x;\lambda,\mu)\right\},
\)
where $d(\cdot)$ is called the dual objective function.
Hence, we call $x\in\cX$ as the primal variable and call $(\lambda,\mu)\in\bR^{m_1+m_2}$ as the dual variable.
Note that the Lagrangian function $\cL$ is convex in $x$ and concave in $(\lambda,\mu)$.
Moreover, for this Lagrangian function,
the subdifferential mapping\footnote{For concave-convex functions, c.f. \cite[p. 374]{rockafellar}.}
$\partial\cL:\cX\times\bR^{m_1}\times\bR^{m_2}\to\cX\times\bR^{m_1}\times\bR^{m_2}$ is defined as follows
\begin{equation*}
\begin{array}{lll}
&(y;u,v)\in\partial \cL(x,\lambda,\mu)
\\[2mm]
&\Leftrightarrow
\left\{
\begin{array}{ll}
\cL(x';\lambda,\mu)\ge \cL(x,\lambda,\mu)+\langle y, x'-x\rangle, &\forall x'\in\cX,
\\[2mm]
\cL(x;\lambda',\mu')\le \cL(x,\lambda,\mu)-\langle u,\lambda'-\lambda\rangle-\langle v,\mu'-\mu\rangle,\quad &\forall (\lambda',\mu')\in\bR^{m_1+m_2}.
\end{array}
\right.
\end{array}
\end{equation*}
Based on the above definition, it is easy to verify that the subdifferential mapping
$\partial\cL$ is a maximal monotone operator.
Moreover, if $(x^*,\lambda^*,\mu^*)\in\cX\times\bR^{m_1}\times\bR^{m_2}$ satisfies
$0\in\partial\cL(x^*,\lambda^*,\mu^*)$, it holds that $x^*$ is an optimal solution to problem \eqref{prob} and $(\lambda^*,\mu^*)$ is an optimal solution to problem \eqref{dual}.
In this case, $(x^*,\lambda^*,\mu^*)$ is called as a saddle point of the Lagrangian function $\cL$, and it holds that $f(x^*)=d(\lambda^*,\mu^*)$.

Since $g$ and $h$ are continuously differentiable, we define
$$
\nabla h(x):=\big(
\nabla h_1(x),\cdots,
\nabla h_{m_1}(x)
\big)
\quad
\mbox{and}
\quad
\nabla g(x):=\big(
\nabla g_1(x),\cdots,\nabla g_{m_2}(x)
\big).
$$
Then, the Karush-Kuhn-Tucker (KKT) system of problem \eqref{prob} is given by
\(
\label{kkt}\left\{
\begin{array}{l}
0\in\partial_x \cL(x,\lambda,\mu):=\partial f(x)+\nabla h(x) \lambda+ \nabla g(x) \mu,\\[2mm]
h(x)=0,\\[2mm]
g(x)\le 0,\ \mu\ge0,\ \langle\mu,g(x)\rangle=0.
\end{array}
\right.
\)
If the solution set to the KKT system \eqref{kkt} is nonempty, from \cite[Theorem 30.4 \& Corollary 30.5.1]{rocbook} one knows that a vector
$(x^*, \lambda^*,\mu^*)\in\cX\times\bR^{m_1}\times\bR^{m_2}$ is a solution to the KKT system \eqref{kkt} if and only if $x^*$ is an optimal solution to problem \eqref{prob} and $(\lambda^*, \mu^*)$ is an optimal solution to problem \eqref{dual}.
Moreover, the solution set to the KKT system \eqref{kkt} can be written as $X^*\times P^*$ with $X^*$ being the solution set to problem \eqref{prob} and $P^*$ being the solution set to problem \eqref{dual}.
In this case, the optimal values of problem \eqref{prob} and problem \eqref{dual}
are equal, and the solution set to the KKT system \eqref{kkt} is exactly the set of saddle points to the Lagrangian function $\cL$.

\section{An Inexact Augmented Larangian Method}

\label{sec:algo}
In this section, we present the inexact augmented Lagrangian method with the practical relative error criterion of \cite{eckstein13}, and summarize its convergence properties from \cite{eckstein13} and \cite{alves}. Some equalities and inequalities, which are useful to study its convergence rate, are also prepared in this section.

We dfine the closed convex cone
$$\cK=\bR^{m_2}_+:=\{v\in\bR^{m_2}:v\ge 0\}.
$$
Let $c>0$ be a given real number. 
The augmented Lagrangian function of problem \eqref{prob} is defined (with $c$ being the penalty parameter) by
\[
\begin{array}{l}
\ds
\cL_c(x,\lambda,\mu):
\ds =f(x)+\langle\lambda, h(x)\rangle+\frac{c}{2}\|h(x)\|^2+\frac{1}{2c}\big(\|\Pi_{\cal K}(\mu+cg(x))\|^2-\|\mu\|^2\big)\\[4mm]
\displaystyle =f(x)+\sum_{i=1}^{m_1}\big(\lambda_ih_i(x)+\frac{c}{2}(h_i(x))^2\big)
+\frac{1}{2c}\sum_{i=1}^{m_2}\Big(\big(\max\{0,\mu_i+cg_i(x)\}\big)^2-\mu_i^2\Big),\quad
\\[4mm]
\hfill \forall x\in\cX, \lambda\in\bR^{m_1}, \mu\in\bR^{m_2}.
\end{array}
\]

When applied to solving problem \eqref{prob}, the augmented Lagrangian method with the practical relative error criterion proposed in \cite[(11)-(15)]{eckstein13} can be described as the following Algorithm.

\begin{center}
\fbox{\begin{minipage}{.97\textwidth}
\begin{algo}
\label{alg:alm}
{\bf An inexact augmented Lagrangian method with a practical relative error criterion for solving problem \eqref{prob}.}
\end{algo}
Let $\sigma\in[0,1)$ and let $\{c_k\}$ be a sequence of positive real numbers such that $\inf_{k\ge 1}\{c_k\}>0$.
Choose $\lambda^0\in\bR^{m_1}$, $\mu^0\in\bR^{m_2}_+$ and $w^0\in\cX$. For $k=1, 2, \ldots$,
\begin{enumerate}
\item[\bf 1.] find $x^{k}\in\cX$ and $y^k\in\cX$ such that
\(
\label{conditionsub}
y^{k}\in\partial_x \cL_{c_k}(x^{k},\lambda^{k-1},\mu^{k-1}),
\)
and
\begin{equation}
\label{condition}
\frac{2}{c_k}
\Big|\langle w^{k-1}-x^{k},y^{k}\rangle\Big|
+
\|y^k\|^2\le \sigma
\left(
\|h(x^k)\|^2
+\left\|\min\left\{\frac{1}{c_k}\mu^{k-1},-g(x^k)\right\}\right\|^2
\right);
\end{equation}
\item[\bf 2.] set
\begin{eqnarray*}
&&\lambda^k:=\lambda^{k-1}+c_kh(x^k),
\\[2mm]
&&\mu^k:=\max\{0,\mu^{k-1}+c_kg(x^k)\},
\\[2mm]
&&w^k:=w^{k-1}-c_k y^k.
\end{eqnarray*}
\end{enumerate}
\end{minipage}}
\end{center}

\begin{remark}
In the above algorithm, each $x^k$ is an approximate solution to the corresponding subproblem with $y^k$ being a subgradient of the corresponding objective function at $x^k$, i.e., $x^k$ approximately solves the problem
$$
\min_{x}\cL_{c_k}(x; \lambda^{k-1},\mu^{k-1}).
$$
If this subproblem admits a solution, say $\tilde x^k$, one can let $y^k=0$ so that \eqref{condition} is satisfied.
Therefore, one can always find the sequence $\{(x^k,y^k)\}$ via Algorithm \ref{alg:alm} if the subproblems are well-defined.

\end{remark}

For the convenience of further discussions, we denote $m:=m_1+m_2$ and define the sequences $\{p^k\}$ and $\{u^k\}$ in $\bR^m$ via
\(
\label{def:pu}
p^k:=(\lambda^k,\mu^k)
\quad\mbox{and}\quad
u^k:=\frac{1}{c_k}(p^{k-1}-p^k).
\)
The following theorem is an immediate extension of \cite[Proposition 1]{eckstein13}.
\begin{theorem}
\label{prop:main}
Assume that the solution set to the KKT system \eqref{kkt} of problem \eqref{prob} is non-empty.
Suppose that the infinite sequences generated by Algorithm \ref{alg:alm} are well-defined. Then,
\begin{enumerate}
\item[(a)] The sequences $\{p^k\}$ and $\{w^k\}$ are bounded;
\item[(b)] $u^k\to 0$ and $y^k\to0$ as $k\to\infty$;
\item[(c)] Any accumulation point of the sequence $\{x^k\}$ is a solutions to problem \eqref{prob}, and any accumulation point of $\{p^k\}$ is a solution to the dual problem \eqref{dual}.

\end{enumerate}
\end{theorem}

As can be observed from Theorem \ref{prop:main}, the convergence results in \cite{eckstein13} show that the sequence $\{p^k\}$ is bounded and each accumulation point of this sequence is a solution to the dual problem \eqref{dual}.
To ensure the full convergence of the sequence $\{p^k\}$, in \cite[Proposition 2]{eckstein13}, the authors introduced the additional criterion to \eqref{condition} that
$$
c_k\|y^k\|\le\xi\|p^{k-1}-p^k\|^2,
$$
where $\xi\ge 0$ is a given constant. As was commented in \cite{eckstein13}, despite $\xi$ can be arbitrarily large, this additional criterion seems stringent. Fortunately, in \cite{alves},
the authors showed that it is not necessary to use this extra criterion to guarantee the convergence of the sequence $\{p^k\}$,
thanks to their insightful investigation of F\'ejer monotone sequences in product spaces.
The following theorem directly comes from \cite[Proposition 2]{alves}.
\begin{theorem}
\label{theo:main}
Suppose that all the assumptions and conditions in Proposition \ref{prop:main} hold.
Then, the whole sequences $\{p^k\}$ converges to a solution to problem \eqref{dual}.
\end{theorem}

As can be seen from Proposition \ref{prop:main} and Theorem \ref{theo:main}, the global convergence of Algorithm \ref{alg:alm} has been well established in the sense that the sequence of the generated dual variables is globally convergent,
which is quite similar to the global convergence properties of the (inexact) augmented Lagrangian method established in \cite{rockafellar}.
In fact, Proposition \ref{prop:main}, together with Theorem \ref{theo:main}, lays the foundation for further investigating the convergence rate of Algorithm \ref{alg:alm} in the next section.

Finally, we should emphasize that, referring to the primal sequence $\{x^k\}$, the best possible result in both \cite{eckstein13} and \cite{alves} is that any accumulation point of this sequence is a solutions to problem \eqref{prob}.

\section{Convergence Rate Analysis}
\label{sec:conv}
This section establishes the local convergence rate of Algorithm \ref{alg:alm} for solving problem \eqref{prob}, under a mild error bound condition.
We first discuss this error bound condition and then given the main result of this paper.

\subsection{An error bound condition}
Let $T:\cH\to\cH$ be a maximal monotone mapping with $T^{-1}$ being its inverse mapping\footnote{One may refer to \cite[Chapter 12]{va} for more information about (maximal) monotone mappings.}.
For solving the general inclusion problem of finding $z\in\cH$ such that
\(
\label{inclu}
0\in T(z),
\)
the (inexact) PPA in \cite{rockafellar76} takes the following iteration scheme
\(
\label{algppa}
z^{k}\approx (I+c_k T)^{-1}(z^{k-1}),\quad k=1,2,\ldots,
\)
where $c_k>0$, and $z^0\in\cH$ is the given initial point.
In \cite{rockafellar76},
the convergence rate of this (inexact) PPA has been analyzed under a local error bound condition in which the solution to \eqref{inclu} should be a singleton. In \cite{luque}, the author extended the convergence rate analysis of \cite{rockafellar76} for the PPA to the case that the solution set of \eqref{inclu} is not necessarily a singleton.
The following definition introduces an error bound condition, which has been used, in the name of a growth condition of maximal monotone operators, in \cite{luque} for the convergence rate analysis of PPA.

\begin{definition}[\cite{robinson}]

The mapping $T^{-1}$ is called locally upper Lipschitz continuous at the origin if $T^{-1}(0)$ is nonempty and there exist constants $\epsilon>0$ and $\kappa>0$ such that
\(
\label{growth}
\dist(z,T^{-1}(0))\le\kappa\|\beta\|,\quad
\forall \beta\in{\mathds B}_{\epsilon}(0)\ \ \mbox{and}\ \ \forall z\in T^{-1}(\beta).
\)
\end{definition}

The augmented Lagrangian method and its inexact version in \cite{rockafellar} can be explained as an application of the PPA to a certain maximal monotone operator\footnote{In particular, for problem \eqref{prob} considered in this paper, this maximal monotone operator is given by $\partial d$  with $d$ being defined in \eqref{dual}. Here the subdifferential mapping is in the concave sense, c.f. \cite[pp. 307--308]{rocbook}.}.
Therefore, the convergence rate analysis in \cite{rockafellar} for the augmented Lagrangian method is heavily dependent on the convergence rate analysis in \cite{rockafellar76} for the PPA.
However, Algorithm \ref{alg:alm} can not be explained as a certain variant of the PPA, at least currently.
Therefore, its convergence rate should be studied via a totally different approach.
In this paper, to analyze the convergence rate of Algorithm \ref{alg:alm}, we make the following assumption, which is an error bound condition on problem \eqref{prob}.

\begin{assumption}
\label{ass:growth}
The mapping $(\partial\cL)^{-1}$ is locally upper Lipschitz continuous at the origin.
\end{assumption}
Before analyzing the local Q-linear convergence rate of Algorithm \ref{alg:alm}, we make the following comment to Assumption \ref{ass:growth}.
\begin{remark}
Assumption \ref{ass:growth} is not a very restrictive condition. On the one hand, it contains the case that $(\partial\cL)^{-1}$ is locally Lipschitz continuous at $0$, which was used extensively in {\rm \cite{rockafellar,rockafellar76}} for deriving the Q-linear convergence rate for PPA and augmented Lagrangian methods.
Note that even the stronger condition than Assumption \ref{ass:growth} that $(\partial\cL)^{-1}$ is locally Lipschitz continuous at $0$ can be satisfied, at least heuristically, for ``most'' convex optimization problems in the form of \eqref{prob} with $f$ being the sum of a twice continuously differentiable convex function plus an indicator function of a closed convex set and $g$ and $h$ being twice continuously differentiable \cite[p. 105, Remark]{rockafellar}.
One may also refer to {\rm\cite[Proposition 4]{rockafellar}} and the discussions after this proposition for more information. In addition, we mention that Assumption \ref{ass:growth} is more general in the sense that $(\partial\cL)^{-1}(0)$ is not necessarily to be a singleton.
\end{remark}

\subsection{The convergence rate}
Now we present the main result of this paper.
\begin{theorem}
\label{thm:main}
Suppose that the solution set $X^*\times P^*$ to the KKT system \eqref{kkt} of problem \eqref{prob} is nonempty and the infinite sequences generated by Algorithm \ref{alg:alm} are well-defined. Suppose that Assumption \ref{ass:growth} holds (with the parameters $\kappa>0$, $\epsilon>0$ such that \eqref{growth} holds for $T=\partial\cL$).
Then, the following statements hold:
\begin{enumerate}
\item[(a)] for any sufficiently large $k$, it holds that
\(
\label{result1}
\dist(x^k,X^*)\le
\frac{\kappa(1+\sqrt{\sigma})}{c_k}\|p^{k-1}-p^k\|,
\)
so that
\(
\label{pconv}
\lim_{k\to\infty}\dist(x^k,X^*)=0;
\)
\item[(b)]
if,
additionally, one has that
\(
\label{cond:c}
\liminf_{k\to\infty}c_k>2\kappa(\sigma+\sqrt{\sigma}),
\)
then, for any sufficiently large $k$
\(
\label{resmain}
\dist(p^{k},P^*)\le \rho_k \dist(p^{k-1},P^*)
\)
with
$$
\rho_k:
=\frac{\kappa\sqrt{1+\sigma}}
{\sqrt{c_k^2-2\kappa(\sigma+\sqrt{\sigma}){c_k}+\kappa^2(1+\sigma)}}
$$
and
\(
\label{limrho}
\limsup_{k\to\infty}\{\rho_k\}
<1.
\)
This means that the local rate of convergence for Algorithm \ref{alg:alm} is Q-linear and the modulus of the convergence rate tends to zero if the sequence $\{c_k\}$ increases to infinity.
\end{enumerate}
\end{theorem}

\begin{proof}
{\it (a)}
According to \eqref{def:pu} and \cite[eq. (24)]{eckstein13}, one has that
$$
\begin{array}{l}
\big(y^k,\frac{1}{c_k}(p^{k-1}-p^k)\big)
=
(y^k,u^k)\in\partial \cL(x^k,p^k).
\end{array}
$$
From \cite[eq. (38)]{eckstein13} we know that, for any $p^*\in P^*$, it holds that
\(
\label{eq:p}
\|p^k-p^*\|=\|p^{k-1}-p^*\|^2-2c_k\langle u^k,p^k-p^*\rangle-\|p^{k-1}-p^k\|^2.
\)
Since $(0,0)\in\partial\cL(x^*,p^*)$ and $\partial\cL$ is maximally monotone, it holds that for any $x^*\in X^*$ and $p^*\in P^*$,
\(
\label{ineq:px}
\langle y^k, x^k-x^*\rangle+\langle u^k, p^k-p^*\rangle\ge 0.
\)
Define the sequences $\{\bar x^k\}$ in $\cX$ and $\{\bar p^k\}$ in $\bR^{m}$ via
\[
\bar x^k:=\Pi_{X^*}(x^k)
\quad\mbox{and}\quad
\bar p^k:=\Pi_{P^*}(p^k).
\]
Combining \eqref{eq:p} and \eqref{ineq:px} together implies that, for any $p^*\in P^*$,
\(
\label{ineq:pin1}
\begin{array}{ll}
&\|p^{k-1}-p^*\|^2
-\|p^{k}-p^*\|^2\\[2mm]
&=\|p^{k}-p^{k-1}\|^2+2c_k
\langle u^k,p^k-p^*\rangle\\[2mm]
&\ge\|p^{k}-p^{k-1}\|^2-2c_k\langle x^k-\bar x^k,y^k\rangle
\\[2mm]
&\ge\|p^{k}-p^{k-1}\|^2-2c_k\|x^k-\bar x^k\|\|y^k\|.
\end{array}
\)
Since Assumption \ref{ass:growth} holds with $\kappa>0$ and $\epsilon>0$,
from \eqref{growth} we know that
$$
\dist\Big((x,p),(X^*,P^*)\Big)\le\kappa\big\|(y,u)\big\|,\quad
\forall (y,u)\in\partial \cL(x,p) \mbox{ with }
\big\|(y,u)\big\|\le\epsilon.
$$
From Proposition \ref{prop:main}(b) and \eqref{ineq:yp1} we know that $\|(y^k,u^k)\|\le\epsilon$ if $k$ is sufficiently large. Then, there exists a positive integer $k_0$ such that for any $k\ge k_0$,
\(
\label{growthbound}
\dist\left((x^{k},p^{k}),(X^*,P^*)\right)
\le
\kappa\left\|(y^k,u^k)\right\|.
\)
This, together with \eqref{def:pu}, implies that
\[
\|x^k-\bar x^k\|^2
\le
\kappa^2\|y^k\|^2+\frac{\kappa^2}{c_k^2}\|p^{k-1}-p^k\|^2.
\]
Hence, it holds that, for any $k\ge k_0$,
\(
\label{ineq:qq}
\|x^k-\bar x^k\|
\le
\kappa\|y^k\|+\frac{\kappa}{c_k}\|p^{k-1}-p^k\|.
\)
Note that \eqref{condition} can be equivalently written as
\(
\label{ineq:yp1}
{2}{c_k}
\big|\langle w^{k-1}-x^{k},y^{k}\rangle\big|
+
c_k^2\|y^k\|^2\le \sigma\|p^{k-1}-p^k\|^2,
\)
which implies that
\(
\label{ineq:yp2}
\|y^k\|\le \frac{\sqrt{\sigma}}{c_k}\|p^{k-1}-p^k\|.
\)
As a result, by combining \eqref{ineq:qq} and \eqref{ineq:yp2} together one gets
\[
\|x^k-\bar x^k\|
\le
\frac{\kappa(1+\sqrt{\sigma})}{c_k}\|p^{k-1}-p^k\|,
\]
so that \eqref{result1} is established. Moreover, since $\inf_{k\ge 0}\{c_k\}>0$,
the above inequality, together with \eqref{def:pu} and Proposition \ref{prop:main}(b), implies \eqref{pconv}.

\medskip
{\it (b)}
One can get from \eqref{ineq:qq} that
\(
\label{ineq:xp1}
\|x^k-\bar x^k\|\|y^k\|\le
\kappa\|y^k\|^2+\frac{\kappa}{c_k}\|y^k\|\|p^{k-1}-p^k\|.
\)
By taking \eqref{ineq:xp1} into \eqref{ineq:pin1}, one has that
\(
\label{ineq:main1}
\begin{array}{ll}
&\|p^{k-1}-p^*\|^2
-\|p^{k}- p^*\|^2
\\[2mm]
&\displaystyle
\ge\|p^{k}-p^{k-1}\|^2-2c_k\left(\kappa\|y^k\|^2+\frac{\kappa}{c_k}\|y^k\|\|p^{k-1}-p^k\|\right)
\\[4mm]
&\displaystyle
=\|p^{k}-p^{k-1}\|^2-2c_k
\kappa\|y^k\|^2-2\kappa\|y^k\|\|p^{k-1}-p^k\|.
\end{array}
\)
Then, by combining \eqref{ineq:yp2} and \eqref{ineq:main1} together we can get that for any $p^*\in P^*$,
\(
\label{ineq:pp2}
\begin{array}{llll}
&\displaystyle
\|p^{k-1}-p^*\|^2
-\|p^{k}-p^*\|^2
\\[3mm]
&\displaystyle
\ge\|p^{k}-p^{k-1}\|^2-2
\kappa\frac{\sigma}{c_k}\|p^{k-1}-p^k\|^2-2\kappa\frac{\sqrt{\sigma}}{c_k}\|p^{k-1}-p^k\|^2
\\[3mm]
&\displaystyle
=\left(
1-\frac{2\kappa(\sigma+\sqrt{\sigma})}{c_k}\right)
\|p^{k}-p^{k-1}\|^2.
\end{array}
\)
On the other hand, as \eqref{growthbound} holds for $k\ge k_0$, from \eqref{ineq:yp2} one has that for any $k\ge k_0$,
\(
\label{ineq:pp3}
\begin{array}{lllll}
\|p^{k}-\bar p^k\|^2
&\ds
\le
\kappa^2\|y^k\|^2+\frac{\kappa^2}{c_k^2}\|p^{k-1}-p^k\|^2
\\[3mm]
&\ds\le
\frac{\kappa^2\sigma}{c^2_k}\|p^{k-1}-p^k\|^2
+\frac{\kappa^2}{c_k^2}\|p^{k-1}-p^k\|^2
\\[3mm]
&\ds
=\frac{(1+\sigma)\kappa^2}{c_k^2}\|p^{k-1}-p^k\|^2.
\end{array}
\)
By combining \eqref{ineq:pp2} and \eqref{ineq:pp3} together, one has that for any $k\ge k_0$,
\(
\label{ineq:final}
\begin{array}{lll}
&\ds
\|p^{k-1}-\bar p^{k-1}\|^2
-\|p^{k}-\bar p^{k}\|^2
\\[3mm]
&\ds
\ge\|p^{k-1}-\bar p^{k-1}\|^2
-\|p^{k}-\bar p^{k-1}\|^2
\\[3mm]
&\ds
\ge\frac{c_k^2-{2\kappa(\sigma+\sqrt{\sigma})}{c_k}}
{\kappa^2(1+\sigma)}
\|p^{k}-\bar p^k\|^2.
\end{array}
\)
According to \eqref{cond:c}, it is easy to see that there exists a certain positive integer $k_0'$ such that for all $k\ge k'_0$, it holds that $c_k>2\kappa(\sigma+\sqrt{\sigma})$,
so that
$$
\frac{c_k^2-{2\kappa(\sigma+\sqrt{\sigma})}{c_k}}
{\kappa^2(1+\sigma)}>0.
$$
Hence, from \eqref{ineq:final} one can get that for any $k\ge\max\{k_0,k'_0\}$,
\[
\frac{\|p^{k}-\bar p^k\|^2}{\|p^{k-1}-\bar p^{k-1}\|^2}
\le
\frac{{\kappa^2(1+\sigma)}}{{\kappa^2(1+\sigma)}
+{c_k^2}-{2\kappa(\sigma+\sqrt{\sigma})}{c_k}}.
\]
This proves \eqref{resmain}.
Finally, denote
\(
\delta:=\liminf_{k\to\infty}\{c_k\}-2\kappa(\sigma+\sqrt{\sigma})>0.
\)
Since $c_k>2\kappa(\sigma+\sqrt{\sigma})$ for $k>k_0'$, one has in this case that
$$
c_k^2-2\kappa(\sigma+\sqrt{\sigma}){c_k}>\delta c_k>2\delta\kappa(\sigma+\sqrt{\sigma}),
$$
so that
$$
\rho_k=
\frac{\kappa\sqrt{1+\sigma}}
{\sqrt{c_k^2-2\kappa(\sigma+\sqrt{\sigma}){c_k}+\kappa^2(1+\sigma)}}
<
\frac{\kappa\sqrt{1+\sigma}}
{\sqrt{2\delta\kappa(\sigma+\sqrt{\sigma})+\kappa^2(1+\sigma)}}<1.
$$
This proves \eqref{limrho} and hence completes the proof of the theorem.
\end{proof}

Before concluding this paper, we should make some comments on the results established above.
\begin{remark}
Apparently,
The result on the rate of convergence established in Theorem \ref{thm:main} depends on the condition \eqref{cond:c}.
This can partially be explained by the fact that due to the subproblems of Algorithm \ref{alg:alm} are solved inexactly, one can not always guarantee the sequence $\{p^k\}$ is F\'ejer monotone to $P^*$ even if $k$ is sufficiently large.
However, this is not something ``bad'' since
\eqref{cond:c} is very likely to be satisfied.
The reason is that, if $c_k$ increases to $+\infty$, \eqref{cond:c} must be satisfied when $k$ is sufficiently large and $\rho_k$ converges to $0$, no matter what the values of $\kappa$ and $\sigma$ are.
As a result, large (but not too large to cause the numerical issues when solving the subproblems) penalty parameters are preferred when implementing Algorithm \ref{alg:alm} as they can both bring better modulus of the convergence rate and enhance the likelihood that \eqref{cond:c} holds.
\end{remark}

\begin{remark}
It was established in {\rm\cite{eckstein13}} that any accumulation point of the primal sequence $\{x^k\}$ is a solution to problem \eqref{prob}, but this sequence is not guaranteed to be bounded.
As a result, the convergence properties of this sequence has not been explored.
Moreover,
even the sequence $\{w^k\}$ is bounded, it is just a sequence of auxiliary variables so that the importance of its convergence properties is not apparent.
Hence, the convergence result of the sequence $\{x^k\}$ established in Theorem \ref{thm:main} (a), to some extent, makes real progress.
\end{remark}

\begin{remark}
Recently, in \cite[Proposition 1]{cui}, the authors have improved the convergence rate analysis of the PPA \eqref{algppa} for solving \eqref{inclu} with the error tolerance criteria in \cite{rockafellar76} by relaxing the upper Lipschitz continuous of $T^{-1}$ at the origin to
the condition that $T^{-1}$ is calm\footnote{
Such a calmness assumption is weaker than the upper Lipschitz continuous property of $T^{-1}$, and one may refer to \cite[Section 2.1]{cui} for details.} at the origin for $z^{\infty}$ with a certain positive modulus, where $z^{\infty}$ is the limit of the sequence ${z^k}$ generated by the PPA (which always exists due to the global convergence property of the PPA).
As can be seen from Proposition \ref{prop:main} and Theorem \ref{theo:main}, the primal sequence generated by Algorithm \ref{alg:alm} is not guaranteed to be convergent in general. Hence, it is impracticable to impose similar calmness assumptions on $(\partial \cL)^{-1}$ instead of using Assumption \ref{ass:growth}.

\end{remark}

\section{Conclusions}
\label{sec:conclusion}
In this paper, we have analyzed the fast Q-linear convergence rate of the inexact augmented Lagrangian method with a practical relative error criterion, under a mild local error bound condition.
The results established in this paper imply that if the penalty parameter increases to infinity, the modulus of the convergence rate would tend to zero, which constitutes an asymptotically superlinear convergence rate.
Besides, without introducing any additional condition, the distance from the primal sequence to the solution set of the primal problem is guaranteed to vanish.
The results here can serve as the guideline for choosing the penalty parameter when implementing the inexact augmented Lagrangian method of \cite{eckstein13}.

\section*{Acknowledgments}
The authors would like to thank Prof. Defeng Sun at the Hong Kong Polytechnic University for bringing the paper of \cite{eckstein13} to their attentions.

\end{document}